\documentclass[a4paper,11pt]{article}%
\usepackage{amsfonts}
\usepackage{amsmath}
\usepackage{amssymb}
\usepackage{graphicx}%
\providecommand{\U}[1]{\protect\rule{.1in}{.1in}}
\newtheorem{theorem}{Theorem}

\newtheorem{lemma}[theorem]{Lemma}

\newtheorem{proposition}[theorem]{Proposition}
\newtheorem{remark}[theorem]{Remark}

\newenvironment{proof}[1][Proof]{\noindent\textbf{#1.} }{\ \rule{0.5em}{0.5em}}
\providecommand{\U}[1]{\protect\rule{.1in}{.1in}}
\begin{document}

\title{Solving an abstract nonlinear eigenvalue problem\\by the inverse iteration method}
\author{Grey Ercole\\{\small \textit{Departamento de Matem\'{a}tica - ICEx, Universidade Federal de
Minas Gerais,}}\\{\small \textit{Av. Ant\^{o}nio Carlos 6627, Caixa Postal 702, 30161-970, Belo
Horizonte, MG, Brazil. }}\\{\small \textit{Email: grey@mat.ufmg.br.}}}
\maketitle

\begin{abstract}
\noindent Let $\left(  X,\left\Vert \cdot\right\Vert _{X}\right)  $ and
$\left(  Y,\left\Vert \cdot\right\Vert _{Y}\right)  $ be Banach spaces over
$\mathbb{R},$ with $X$ uniformly convex and compactly embedded into $Y.$ The
inverse iteration method is applied to solve the abstract eigenvalue problem
$A(w)=\lambda\left\Vert w\right\Vert _{Y}^{p-q}B(w),$ where the maps
$A:X\rightarrow X^{\star}$ and $B:Y\rightarrow Y^{\star}$ are homogeneous of
degrees $p-1$ and $q-1,$ respectively.

\end{abstract}

\noindent\textbf{2010 AMS Classification.} 47J05; 47J25; 35P30.


\section{Introduction}

Many eigenvalue-type problems involving quasilinear elliptic equations are
formulated as a functional equation of the form%
\begin{equation}
A(w)=\lambda\left\Vert w\right\Vert _{Y}^{p-q}B(w), \label{abst}%
\end{equation}
where $\left(  X,\left\Vert \cdot\right\Vert _{X}\right)  $ and $\left(
Y,\left\Vert \cdot\right\Vert _{Y}\right)  $ are Banach spaces over
$\mathbb{R},$ $X$ compactly embedded into $Y,$ and the maps $A:X\rightarrow
X^{\star}$ and $B:Y\rightarrow Y^{\star}$ are homogeneous of degrees $p-1>0$
and $q-1>0,$ respectively, that is:

\begin{enumerate}
\item[(A1)] $A(tw)=\left\vert t\right\vert ^{p-2}tA(w)$ for all $t\in
\mathbb{R};$

\item[(B1)] $B(tw)=\left\vert t\right\vert ^{q-2}tB(w)$ for all $t\in
\mathbb{R}.$
\end{enumerate}

We say that the pair $\left(  \lambda,w\right)  \in\mathbb{R}\times
X\setminus\left\{  0\right\}  $ solves (\ref{abst}) if, and only if,
\begin{equation}
\left\langle A(w),v\right\rangle =\lambda\left\Vert w\right\Vert _{Y}%
^{p-q}\left\langle B(w),v\right\rangle ,\quad\forall\mathrm{\;}v\in X,
\label{absol}%
\end{equation}
where we are using the notation $\left\langle f,v\right\rangle \overset
{\mathrm{def}}{=}f(v).$

In this paper we apply the inverse iteration method to solve the abstract
equation (\ref{abst}) by assuming the following additional hypotheses on the
maps $A$ and $B:$

\begin{enumerate}
\item[(A2)] $\left\langle A(u),v\right\rangle \leq\left\Vert u\right\Vert
_{X}^{p-1}\left\Vert v\right\Vert _{X}$ for all $u,v\in X,$ with the equality
occurring if, and only if, either $u=0$ or $v=0$ or $u=tv,$ for some $t>0;$

\item[(B2)] $\left\langle B(u),v\right\rangle \leq\left\Vert u\right\Vert
_{Y}^{q-1}\left\Vert v\right\Vert _{Y}$ for all $u,v\in Y,$ with the equality
occurring whenever $u=tv,$ for some $t\geq0;$

\item[(AB)] for each $w\in Y\setminus\left\{  0\right\}  $ given, there exists
at least one $u\in X\setminus\left\{  0\right\}  $ such that%
\[
\left\langle A(u),v\right\rangle =\left\langle B(w),v\right\rangle
,\quad\forall\mathrm{\;}v\in X.
\]

\end{enumerate}

We observe from (\textrm{A1}) and (\textrm{B1}) that (\ref{abst}) is
homogeneous, that is: if $\left(  \lambda,w\right)  $ solves (\ref{abst}) the
same holds true for $\left(  \lambda,tw\right)  ,$ for all $t\not =0.$
Motivated by this intrinsic property of eigenvalue problems, we say that
$\lambda$ is an \textit{eigenvalue} of (\ref{abst}) and that $w$ is an
\textit{eigenvector} of (\ref{abst}) corresponding to $\lambda$ or, for
shortness, we simply say that $\left(  \lambda,w\right)  $ is an
\textit{eigenpair} of (\ref{abst}).

Hypotheses (\textrm{A2}) and (\textrm{B2}) imply, respectively, that
\begin{equation}
\left\langle A(w),w\right\rangle =\left\Vert w\right\Vert _{X}^{p}%
,\quad\forall\mathrm{\;}w\in X \label{A2a}%
\end{equation}
and%
\begin{equation}
\left\langle B(w),w\right\rangle =\left\Vert w\right\Vert _{Y}^{q}%
,\quad\forall\mathrm{\;}w\in Y. \label{B2a}%
\end{equation}
Thus, by choosing $v=w$ in (\ref{absol}), we see that
\[
\lambda=\frac{\left\Vert w\right\Vert _{X}^{p}}{\left\Vert w\right\Vert
_{Y}^{p}},
\]
which shows that the eigenvalues of (\ref{abst}) are nonnegative. Actually,
they are bounded from below by
\[
\mu:=\inf\left\{  \left\Vert w\right\Vert _{X}^{p}:w\in X\cap\mathbb{S}%
_{Y}\right\}  ,
\]
where $\mathbb{S}_{Y}:=\left\{  w\in Y:\left\Vert w\right\Vert _{Y}=1\right\}
$ is the unit sphere in $Y.$

We note that the compactness of the embedding $X\hookrightarrow Y$ (which we
are assuming in this paper) implies that $\mu$ is positive and reached in
$\mathbb{S}_{Y}.$ Moreover, assuming in addition the conditions (\textrm{A1}),
(\textrm{A2}), (\textrm{B1}), (\textrm{B2}) and (\textrm{AB}) we will show
(see Proposition \ref{mmu}) that $\mu$ is an eigenvalue and that its
corresponding eigenvectors are precisely the scalar multiple of those vectors
in $\mathbb{S}_{Y}$ at which $\mu$ is reached. Because of this, we refer to
$\mu$ as the \textit{first eigenvalue} of (\ref{abst}) and any of its
corresponding eigenvectors as a \textit{first eigenvector}.

As we will see, hypothesis (\textrm{AB}) allows us to construct, for each
$w_{0}\in\mathbb{S}_{Y},$ an inverse iteration sequence $\left\{  w_{0}%
,w_{1},w_{2},\ldots\right\}  \subset X\cap\mathbb{S}_{Y}$ satisfying%
\[
\left\langle A(w_{n+1}),v\right\rangle =\lambda_{n}\left\langle B(w_{n}%
),v\right\rangle ,\quad\forall\mathrm{\;}v\in X
\]
where $\lambda_{n}\geq\mu.$

Our main result in this paper is stated as follows.

\begin{theorem}
\label{maintheo}Assume that $X$ is uniformly convex and compactly embedded
into $Y,$ and that the maps $A:X\rightarrow X^{\star}$ and $B:Y\rightarrow
Y^{\star}$ are continuous and satisfy the hypotheses (\textrm{A1}),
(\textrm{A2}), (\textrm{B1}), (\textrm{B2}) and (\textrm{AB}). The sequences
$\left\{  \lambda_{n}\right\}  _{n\in\mathbb{N}}$ and $\left\{  \left\Vert
w_{n+1}\right\Vert _{X}^{p}\right\}  _{n\in\mathbb{N}}$ are nonincreasing and
converge to the same limit $\lambda,$ which is bounded from below by $\mu.$
Moreover, $\lambda$ is an eigenvalue of (\ref{abst}) and there exists a
subsequence $\left\{  n_{j}\right\}  _{j\in\mathbb{N}}$ such that both
$\left\{  w_{n_{j}}\right\}  _{j\in\mathbb{N}}$ and $\left\{  w_{n_{j}%
+1}\right\}  _{j\in\mathbb{N}}$converge in $X$ to the same vector $w\in
X\cap\mathbb{S}_{Y},$ which is an eigenvector corresponding to $\lambda.$
\end{theorem}

The proof of this result is presented in Section \ref{sec1} by combining two
lemmas. In Lemma \ref{HyndLind} we obtain, from the hypotheses (\textrm{A2})
and (\textrm{B2}), the monotonicity of the sequences $\left\{  \lambda
_{n}\right\}  _{n\in\mathbb{N}}$ and $\left\{  \left\Vert w_{n+1}\right\Vert
_{X}^{p}\right\}  _{n\in\mathbb{N}}$ as well as their convergence to
$\lambda.$

In Lemma \ref{main} we use the uniform convexity of $X$ and the compactness of
the embedding $X\hookrightarrow Y$ to guarantee the existence of a subsequence
$\left\{  w_{n_{j}}\right\}  _{j\in\mathbb{N}}$ converging in $X$ to a
function $w\in X\cap\mathbb{S}_{Y}.$ A delicate issue in the conclusion of the
lemma is to show that the subsequence $\left\{  w_{n_{j}+1}\right\}
_{j\in\mathbb{N}}$ also converges to $w$ in order to pass to the limit, as
$j\rightarrow\infty,$ in
\[
\left\langle A(w_{n_{j}+1}),v\right\rangle =\lambda_{n_{j}}\left\langle
B(w_{n_{j}}),v\right\rangle ,\quad\forall\mathrm{\;}v\in X.
\]
For this we use the hypothesis (\textrm{A2}), which plays the same structural
role that the H\"{o}lder's inequality plays in the quasilinear elliptic
problems (we recall that the equality in the H\"{o}lder's inequality implies
that the functions involved, raised to conjugate exponents, are proportional).

We conclude Section \ref{sec1} by remarking that when $\lambda$ is simple,
meaning that its corresponding eigenvectors are scalar multiple of each other,
then $w$ and $-w$ are the only cluster points of the sequence $\left\{
w_{n}\right\}  _{n\in\mathbb{N}}.$ Thus, in some concrete situations a
suitable choice of $w_{0}\in\mathbb{S}_{Y}$ guarantees that the vector $w$ is
the only cluster point of $\left\{  w_{n}\right\}  _{n\in\mathbb{N}}$ when
$\lambda$ is known to be simple. Thus, in such a situation one has
$w_{n}\rightarrow w.$

A prototype for (\ref{abst}) is the following Dirichlet problem in a bounded
domain $\Omega\ $of $\mathbb{R}^{N},$ $N\geq2:$%
\begin{equation}
\left\{
\begin{array}
[c]{lll}%
-\Delta_{p}u=\lambda\left\Vert u\right\Vert _{q}^{p-q}\left\vert u\right\vert
^{q-2}u & \mathrm{in} & \Omega\\
u=0 & \mathrm{on} & \partial\Omega,
\end{array}
\right.  \label{pq}%
\end{equation}
where $\Delta_{p}u:=\operatorname{div}\left(  \left\vert \nabla u\right\vert
^{p-2}\nabla u\right)  $ is the $p$-Laplacian operator,
\[
1\leq q<p^{\star}:=\left\{
\begin{array}
[c]{lll}%
\frac{Np}{N-p} & \mathrm{if} & 1<p<N\\
\infty & \mathrm{if} & p\geq N,
\end{array}
\right.
\]
and $\left\Vert \cdot\right\Vert _{r}$ denotes the norm of $L^{r}(\Omega)$ for
$1\leq r\leq\infty$ (we will use this notation from now on).

Indeed, the concept of (weak) solution for (\ref{pq}) takes the form of
(\ref{absol}) with%
\begin{equation}
\left\langle A(u),v\right\rangle :={\int_{\Omega}}\left\vert \nabla
u\right\vert ^{p-2}\nabla u\cdot\nabla v\mathrm{d}x\quad\mathrm{and}%
\quad\left\langle B(u),v\right\rangle :={\int_{\Omega}}\left\vert u\right\vert
^{q-2}uv\mathrm{d}x. \label{A=W}%
\end{equation}
In this setting, $Y$ is the Lebesgue space $L^{q}(\Omega)$ endowed with the
standard norm%
\[
\left\Vert u\right\Vert _{q}:=\left(  {\int_{\Omega}}\left\vert u\right\vert
^{q}\mathrm{d}x\right)  ^{\frac{1}{q}}%
\]
and $X$ is the Sobolev space
\[
W_{0}^{1,p}(\Omega):=\left\{  u\in L^{p}(\Omega):\nabla u\in L^{p}(\Omega
)^{N}\,\mathrm{and}\,u=0\,\mathrm{on}\,\partial\Omega\right\}  ,
\]
endowed with the norm%
\[
\left\Vert u\right\Vert _{W_{0}^{1,p}}:=\left\Vert \nabla u\right\Vert
_{p}=\left(  {\int_{\Omega}}\left\vert \nabla u\right\vert ^{p}\mathrm{d}%
x\right)  ^{\frac{1}{p}}%
\]
which makes $W_{0}^{1,p}(\Omega)$ an uniformly convex Banach space.

The hypotheses (\textrm{A1}) and (\textrm{B1}) can be easily checked for the
maps $A$ and $B$ defined in (\ref{A=W}), whereas (\textrm{A2}) and
(\textrm{B2}) are deduced from H\"{o}lder's inequality. The compactness of the
Sobolev embedding $W_{0}^{1,p}(\Omega)\hookrightarrow L^{q}(\Omega),$ for
$1\leq q<p^{\star},$ is a well-known fact as well as the continuity of the
functions $A:W_{0}^{1,p}(\Omega)\rightarrow W_{0}^{-1,p^{\prime}}(\Omega)$ and
$B:L^{q}(\Omega)\rightarrow L^{q^{\prime}}(\Omega).$ (It is usual to denote
the dual space of $W_{0}^{1,p}(\Omega)$ by $W_{0}^{-1,p^{\prime}}(\Omega),$
where $r^{\prime}:=\frac{r}{r-1}$ is the H\"{o}lder conjugate of $r>1,$ i. e.
$\frac{1}{r}+\frac{1}{r^{\prime}}=1.$)

Property (\textrm{AB}) also holds true since $A$ is surjective and $B(w)\in
W_{0}^{-1,p^{\prime}}(\Omega)$ for all $w\in L^{q}(\Omega).$ We refer the
reader to \cite{mawhin,Le} where all the properties are proved.

The following facts regarding the eigenvalue problem (\ref{pq}) are well-known
(see \cite{Peral, FL,IO,Otani1}): there exists a sequence of eigenvalues
tending to $\infty;$ every eigenfunction belongs to $L^{\infty}(\Omega)$ and
the first eigenfunctions do not change sign in $\Omega.$ (When dealing with
spaces of functions, the nomenclature "eigenfunction" seems to be more
appropriate than "eigenvector".)

The particular case $q=p,$
\begin{equation}
\left\{
\begin{array}
[c]{lll}%
-\Delta_{p}u=\lambda\left\vert u\right\vert ^{p-2}u & \mathrm{in} & \Omega\\
u=0 & \mathrm{on} & \partial\Omega,
\end{array}
\right.  \label{pp}%
\end{equation}
has been extensively studied over the last three decades. Its first eigenvalue%
\begin{equation}
\lambda_{p}:=\min\left\{  \left\Vert \nabla u\right\Vert _{p}^{p}:u\in
W_{0}^{1,p}(\Omega)\quad\mathrm{and}\quad\left\Vert u\right\Vert
_{p}=1\right\}  \label{lambp}%
\end{equation}
is isolated and simple. Moreover, the first eigenfunctions are the only
eigenfunctions that do not change sign in $\Omega.$ These and other properties
of (\ref{pp}) can be verified in \cite{Anane, Lindqvist, Otani2} and
references therein.

When $p\not =2$ the eigenvalue problem (\ref{pp}) is very difficult to be
solved analytically and even numerically, since it loses the linear character
of $p=2$ and acquires the singular or degenerate term $\left\vert \nabla
u\right\vert ^{p-2}.$ We remark that analytical expressions for the first
eigenvalue are not known in general, not even for simple domains, such as
squares, balls or triangles.

In \cite{JFA} the inverse iteration method was introduced to solve (\ref{pp})
in the particular domain: the unit ball $B_{1}:=\left\{  x\in\mathbb{R}%
^{N}:\left\vert x\right\vert =1\right\}  .$ Starting with $u_{0}\equiv1$ and
exploring the radial structure of the Dirichlet problem at each iteration
step, the authors proved that
\[
\lim_{n\rightarrow\infty}\left(  \frac{\left\Vert u_{n}\right\Vert _{\infty}%
}{\left\Vert u_{n+1}\right\Vert _{\infty}}\right)  ^{p-1}=\lambda_{p}%
\quad\mathrm{and}\quad\lim_{n\rightarrow\infty}\frac{u_{n}}{\left\Vert
u_{n}\right\Vert _{\infty}}=u_{p}\quad\mathrm{in}\quad C^{1}(\overline{B_{1}%
}),
\]
where $\left\Vert \cdot\right\Vert _{\infty}$ denotes the sup norm and $u_{p}$
denotes the positive first eigenfunction such that $\left\Vert u_{p}%
\right\Vert _{\infty}=1.$

They also conjectured that%
\begin{equation}
\lim_{n\rightarrow\infty}\left(  \frac{\left\Vert u_{n}\right\Vert _{p}%
}{\left\Vert u_{n+1}\right\Vert _{p}}\right)  ^{p-1}=\lambda_{p}
\label{conjec}%
\end{equation}
for a general bounded domain and presented some numerical experiments for the
unit square as motivation to their conjecture.

The approach used in \cite{JFA}, based on radial symmetry, was adapted in
\cite{JMAA} to obtain the pair $\left(  \lambda_{p},u_{p}\right)  $ for a
radially symmetric annulus.

Recently, in \cite{PAMS}, the authors considered, for a general bounded domain
$\Omega,$ the sequence of iterates
\[
\psi_{n}:=(\lambda_{p})^{\frac{n}{p-1}}u_{n}%
\]
where $u_{0}\in L^{p}(\Omega)$ is given and $-\Delta_{p}u_{n+1}=\left\vert
u_{n}\right\vert ^{p-2}u_{n}.$ By making use of the minimizing property
(\ref{lambp}) of $\lambda_{p}$, they proved the convergence, in $W_{0}%
^{1,p}(\Omega),$ of the sequence $\left(  \psi_{n}\right)  _{n\in\mathbb{N}}$
to a function $\psi.$ Then, under the assumption $\psi\not \equiv 0,$ they
concluded that $\psi$ is a first eigenfunction and proved the conjecture
(\ref{conjec}) posed in \cite{JFA}. They also showed that $\psi\not \equiv 0$
if either $u_{0}\geq ke_{p}$ for some positive constant $k$ or $u_{0}\geq0,$
$u_{0}\not \equiv 0$ and $\Omega$ is sufficiently smooth. (Here $e_{p}$
denotes the positive eigenfunction such that $\left\Vert e_{p}\right\Vert
_{p}=1.$) It is simple to check that $u_{0}\equiv1$ leads to $\psi
\not \equiv 0$.

We emphasize that the minimizing property (\ref{lambp}) of the first
eigenvalue $\lambda_{p}$ plays a decisive role in the approach of \cite{PAMS}
and makes it applicable only to this eigenvalue.

The literature on the eigenvalue problem (\ref{pq}) in the case $q\not =p$
(which is shorter than in the case $q=p$), shows that there are some
differences between the cases $1\leq q<p$ and $p<q<p^{\star}$ with respect to
the properties of the first eigenvalue%
\begin{equation}
\lambda_{q}:=\min\left\{  \left\Vert \nabla u\right\Vert _{p}^{p}:u\in
W_{0}^{1,p}(\Omega)\quad\mathrm{and}\quad\left\Vert u\right\Vert
_{q}=1\right\}  \label{lamdaq}%
\end{equation}
(see \cite{Peral, ADMA, FL,Kaw,Nazarov}).

In the cases $1\leq q<p$ and $q=p$ some properties of the first eigenvalue
problem are shared. For example, the first eigenvalue is simple and the first
eigenfunctions are the only that do not change sign in $\Omega.$ Because of
these properties, we can guarantee that our method is successful when it is
used for the purpose of achieving a first eigenpair. In fact, if $1\leq q\leq
p$ and $u_{0}\in L^{q}(\Omega)\setminus\left\{  0\right\}  $ is nonnegative,
then $\lambda_{n}\rightarrow\lambda_{q}$ and $w_{n}\rightarrow e_{q}$ where
$e_{q}$ is the positive $L^{q}$-normalized eigenfunction (it is not necessary
to pass to a subsequence).

When $p<q<p^{\star}$ and $\Omega$ is a general bounded domain the simplicity
of $\lambda_{q}$ is not guaranteed nor the exclusivity of the first
eigenfunctions with respect to have a definite sign. Thus, in this situation,
we cannot guarantee that the eigenvalue $\lambda,$ obtained when $u_{0}$ is
nonnegative, coincides with $\lambda_{q}.$ By the way, we think that our
method could be used to investigate, at least numerically, the existence of
positive eigenfunctions associated with $\lambda>\lambda_{q}$ for some domains.

Our first motivation, inspired by the papers \cite{JFA,PAMS}, was to apply the
inverse iteration method to (\ref{pq}). However, we realized that the
arguments we had developed to deal with this problem depend only on the
properties of the functions $A$ and $B$ defined in (\ref{A=W}) combined with
compactness. Thus, we arrived at the abstract eigenvalue problem (\ref{abst})
under the hypotheses (\textrm{A1}), (\textrm{A2}), (\textrm{B1}),
(\textrm{B2}) and (\textrm{AB}).

We would like to emphasize that our abstract approach covers a large range of
eigenvalue problems involving partial differential equations of quasilinear
elliptic type and serves as a theoretical basis for a numerical treatment of
them. For the sake of completeness, we present in Section \ref{sec2} two more
examples of such problems: the Dirichlet eigenproblem for the $s$-fractional
$p$-Laplacian and a Steklov-type eigenvalue problem for the $p$-Laplacian
involving a homogenous term of degree $q-1$ on the boundary.

Our results in this paper complement those of \cite{HL2}. In the first part of
that paper the authors extend their own results presented in \cite{PAMS} to an
abstract setting, aiming to approximate the least Rayleigh quotient
$\Phi(u)/\left\Vert u\right\Vert _{Y}^{p}$ where, according to our notation,
$\Phi:X\rightarrow\lbrack0,\infty]$ is a functional satisfying certain
properties (among them, strict convexity and positive homogeneity of degree
$p>1$) and $X:=\left\{  u\in Y:\Phi(u)<\infty\right\}  .$ The authors reduce
the problem of minimizing the Rayleigh quotient above to an equivalent
subdifferential equation involving the subdifferentials of both functionals
$\Phi$ and $\frac{1}{p}\left\Vert \cdot\right\Vert _{Y}^{p}.$ Then, they apply
an inverse iteration scheme to solve the subdifferential equation. Our
approach, however, embraces eigenvalue problems that are not necessarily
linked to least Rayleigh quotients. Moreover, it guarantees that the inverse
iteration sequence always produces an eigenvalue.

\section{The results of convergence\label{sec1}}

In this section we assume that $X$ is uniformly convex, compactly embedded
into $Y$ and that $A:X\rightarrow X^{\star}$ and $B:Y\rightarrow Y^{\star}$
are continuous maps satisfying the hypotheses (\textrm{A1}), (\textrm{A2}),
(\textrm{B1}), (\textrm{B2}) and (\textrm{AB}), stated in the Introduction.

We recall that
\begin{equation}
\mu:=\inf\left\{  \left\Vert w\right\Vert _{X}^{p}:w\in X\cap\mathbb{S}%
_{Y}\right\}  , \label{mu}%
\end{equation}
where $\mathbb{S}_{Y}:=\left\{  w\in Y:\left\Vert w\right\Vert _{Y}=1\right\}
.$

\begin{proposition}
\label{mmu}Let $\left\{  w_{n}\right\}  _{n\in\mathbb{N}}\subset
X\cap\mathbb{S}_{Y}$ be a minimizing sequence of (\ref{mu}), that is:
$\left\Vert w_{n}\right\Vert _{Y}=1$ and $\left\Vert w_{n}\right\Vert
_{X}\rightarrow\mu.$ There exist a subsequence $\left\{  w_{n_{j}}\right\}
_{j\in\mathbb{N}}$ converging weakly in $X$ to a vector $w\in X\cap
\mathbb{S}_{Y}$ which reaches $\mu$ (i.e. $\left\Vert w\right\Vert _{Y}=1$ and
$\left\Vert w\right\Vert _{X}=\mu$). Moreover, $\mu$ is an eigenvalue of
(\ref{abst}) and its corresponding eigenvectors are precisely the scalar
multiple of those vectors where $\mu$ is reached.
\end{proposition}

\begin{proof}
Since $\left\{  w_{n}\right\}  _{n\in\mathbb{N}}$ is a bounded sequence in
$X,$ there exist a subsequence $\left\{  w_{n_{j}}\right\}  _{j\in\mathbb{N}}$
and a vector $w\in X$ such that $w_{n_{j}}\rightharpoonup w$ in $X$ (weak
convergence) and $w_{n_{j}}\rightarrow w$ in $Y$. Here we have used that $X$
is reflexive and compactly embedded into $Y.$ The convergence $w_{n_{j}%
}\rightarrow w$ in $Y$ implies that $\left\Vert w\right\Vert _{Y}%
=\lim_{j\rightarrow\infty}\left\Vert w_{n_{j}}\right\Vert _{Y}=1,$ whereas the
weak convergence $w_{n_{j}}\rightharpoonup w$ in $X$ yields%
\[
\left\Vert w\right\Vert _{X}\leq\lim_{j\rightarrow\infty}\left\Vert w_{n_{j}%
}\right\Vert _{X}=\mu^{1/p}.
\]
Thus, since $\mu\leq\left\Vert w\right\Vert _{X}^{p}$ we conclude that
$\mu=\left\Vert w\right\Vert _{X}^{p}.$

Now, let us prove that the pair $\left(  \mu,w\right)  $ solves (\ref{abst}).
According to (\textrm{AB}) there exists $u\in X\setminus\left\{  0\right\}  $
such that
\[
\left\langle A(u),v\right\rangle =\left\langle B(w),v\right\rangle
,\quad\forall\mathrm{\;}v\in X.
\]
In view of (\textrm{A1}) we can rewrite this equation as%
\begin{equation}
\left\langle A(\widetilde{w}),v\right\rangle =\gamma\left\langle
B(w),v\right\rangle ,\quad\forall\mathrm{\;}v\in X \label{1b}%
\end{equation}
where $\gamma:=\left\Vert u\right\Vert _{Y}^{1-p}$ and $\widetilde
{w}:=\left\Vert u\right\Vert _{Y}^{-1}u$ (so that $\widetilde{w}\in
X\cap\mathbb{S}_{Y}$). Taking $v=w$ in (\ref{1b}) we obtain, from (\ref{B2a})
and (\textrm{A2})%
\begin{equation}
\gamma=\gamma\left\Vert w\right\Vert _{Y}^{q}=\gamma\left\langle
B(w),w\right\rangle =\left\langle A(\widetilde{w}),w\right\rangle
\leq\left\Vert \widetilde{w}\right\Vert _{X}^{p-1}\left\Vert w\right\Vert
_{X}=\left\Vert \widetilde{w}\right\Vert _{X}^{p-1}\mu^{1/p} \label{1a}%
\end{equation}
and taking $v=\widetilde{w}$ in (\ref{1b}) we obtain, from (\ref{A2a}) and
(\textrm{B2})%
\[
\left\Vert \widetilde{w}\right\Vert _{X}^{p}=\left\langle A(\widetilde
{w}),\widetilde{w}\right\rangle =\gamma\left\langle B(w),\widetilde
{w}\right\rangle \leq\gamma\left\Vert w\right\Vert _{Y}^{q-1}\left\Vert
\widetilde{w}\right\Vert _{Y}=\gamma.
\]
It follows that%
\[
\mu\leq\left\Vert \widetilde{w}\right\Vert _{X}^{p}\leq\gamma\leq\left\Vert
\widetilde{w}\right\Vert _{X}^{p-1}\mu^{1/p},
\]
where the first inequality comes from the definition of $\mu.$ A simple
analysis of these latter inequalities shows that all of them are, in fact,
equalities. Thus,
\[
\mu=\left\Vert \widetilde{w}\right\Vert _{X}^{p}\quad\mathrm{and}\quad
\gamma=\mu.
\]
Hence, (\ref{1a}) implies that%
\[
\left\langle A(\widetilde{w}),w\right\rangle =\left\Vert \widetilde
{w}\right\Vert _{X}^{p-1}\left\Vert w\right\Vert _{X}%
\]
and then (\textrm{A2}) leads us to conclude that $\widetilde{w}=w$ (note that
$\left\Vert \widetilde{w}\right\Vert _{X}=\left\Vert w\right\Vert _{X}=\mu$).
Thus, (\ref{1b}) yields%
\[
\left\langle A(w),v\right\rangle =\mu\left\langle B(w),v\right\rangle
,\quad\forall\mathrm{\;}v\in X,
\]
showing that $\left(  \mu,w\right)  $ is an eigenpair. Repeating the same
arguments we can see that any other vector at which $\mu$ is reached is also
an eigenvector corresponding to $\mu.$ In order to complete this proof we
observe that if $u\in X$ is an eigenvector corresponding to $\mu$ then $u=tw$
for some $w\in X\cap\mathbb{S}_{Y}$ such that $\left\Vert w\right\Vert
_{X}^{p}=\mu.$ Indeed, we can pick $t=\left\Vert u\right\Vert _{Y}$ and
$w=\left\Vert u\right\Vert _{Y}^{-1}u\in X\cap\mathbb{S}_{Y},$ since
\[
\left\Vert u\right\Vert _{X}^{p}=\left\langle A(u),u\right\rangle
=\mu\left\Vert u\right\Vert _{Y}^{p-q}\left\langle B(u),u\right\rangle
=\mu\left\Vert u\right\Vert _{Y}^{p}%
\]
implies that $\left\Vert w\right\Vert _{X}^{p}=\mu.$
\end{proof}

\begin{remark}
The previous proof does not require of $X$ to be uniformly convex. In fact,
reflexivity is enough. However, when $X$ is uniformly convex the minimizing
subsequence $\left\{  w_{n_{j}}\right\}  _{j\in\mathbb{N}}$ converges strongly
to $w,$ since $\left\Vert w\right\Vert _{X}=\lim_{j\rightarrow\infty
}\left\Vert w_{n_{j}}\right\Vert _{X}.$
\end{remark}

Now, let us fix an arbitrary vector $w_{0}\in\mathbb{S}_{Y}.$ Thanks to
property (\textrm{AB}), there exists $u_{1}\in X\setminus\left\{  0\right\}  $
such that%
\[
\left\langle A(u_{1}),v\right\rangle =\left\langle B(w_{0}),v\right\rangle
,\quad\forall\mathrm{\;}v\in X.
\]
Hence, by multiplying this equation by $\left\Vert u_{1}\right\Vert _{Y}%
^{1-p}$ and setting
\[
w_{1}:=\left\Vert u_{1}\right\Vert _{Y}^{-1}u_{1}\quad\mathrm{and}\quad
\lambda_{1}:=(\left\Vert u_{1}\right\Vert _{Y})^{1-p}%
\]
we obtain%
\[
\left\langle A(w_{1}),v\right\rangle =\lambda_{1}\left\langle B(w_{0}%
),v\right\rangle ,\quad\forall\mathrm{\;}v\in X.
\]

Repeating inductively the above argument we construct the iteration sequence
$\left\{  w_{n}\right\}  _{n\in\mathbb{N}}\subset X\cap\mathbb{S}_{Y}%
\ $satisfying%
\begin{equation}
\left\langle A(w_{n+1}),v\right\rangle =\lambda_{n}\left\langle B(w_{n}%
),v\right\rangle ,\quad\forall\mathrm{\;}v\in X, \label{wn}%
\end{equation}
where $\lambda_{n}=\left\Vert u_{n+1}\right\Vert _{Y}^{1-p}.$

We observe that
\begin{equation}
\lambda_{n}\geq\mu,\quad\forall\mathrm{\;}n\in\mathbb{N}. \label{mulan}%
\end{equation}
Indeed, since both $w_{n}$ and $w_{n+1}$ belong to $\mathbb{S}_{Y}$, by taking
$v=w_{n+1}$ in (\ref{wn}) and using the definition of $\mu$ we find%
\begin{equation}
\mu\leq\left\Vert w_{n+1}\right\Vert _{X}^{p}=\left\langle A(w_{n+1}%
),w_{n+1}\right\rangle =\lambda_{n}\left\langle B(w_{n}),w_{n+1}\right\rangle
\leq\lambda_{n}\left\Vert w_{n}\right\Vert _{Y}^{q-1}\left\Vert w_{n+1}%
\right\Vert _{Y}=\lambda_{n}. \label{2a}%
\end{equation}

\begin{lemma}
\label{HyndLind}The sequences $\left\{  \lambda_{n}\right\}  _{n\in\mathbb{N}%
}$ and $\left\{  \left\Vert w_{n+1}\right\Vert _{X}^{p}\right\}
_{n\in\mathbb{N}}$ are nonincreasing and converge to the same limit $\lambda.$
Moreover,
\begin{equation}
\lambda\geq\mu. \label{2}%
\end{equation}

\end{lemma}

\begin{proof}
We can see from (\ref{2a}) that
\[
\left\Vert w_{n+1}\right\Vert _{X}^{p}\leq\lambda_{n}\quad\forall
\mathrm{\;}n\in\mathbb{N}.
\]
Taking $v=w_{n}$ in (\ref{wn}) we have%
\[
\lambda_{n}=\lambda_{n}\left\Vert w_{n}\right\Vert _{Y}^{q}=\lambda
_{n}\left\langle B(w_{n}),w_{n}\right\rangle =\left\langle A(w_{n+1}%
),w_{n}\right\rangle \leq\left\Vert w_{n+1}\right\Vert _{X}^{p-1}\left\Vert
w_{n}\right\Vert _{X}.
\]
Hence,%
\begin{equation}
\left\Vert w_{n+1}\right\Vert _{X}^{p}\leq\lambda_{n}\leq\left\Vert
w_{n+1}\right\Vert _{X}^{p-1}\left\Vert w_{n}\right\Vert _{X}\leq(\lambda
_{n})^{\frac{p-1}{p}}(\lambda_{n-1})^{\frac{1}{p}}, \label{1}%
\end{equation}
from which we obtain%
\[
\left\Vert w_{n+1}\right\Vert _{X}\leq\left\Vert w_{n}\right\Vert _{X}%
\quad\mathrm{and}\quad\lambda_{n}\leq\lambda_{n-1}.
\]
Since the numerical sequences $\left\{  \lambda_{n}\right\}  _{n\in\mathbb{N}%
}$ and $\left\{  \left\Vert w_{n+1}\right\Vert _{X}^{p}\right\}
_{n\in\mathbb{N}}$ are also bounded from below they are convergent. Thus, by
making $n\rightarrow\infty$ in (\ref{1}) we can see that both converge to the
same limit, which we denote by $\lambda.$ The inequality (\ref{2}) follows
directly from (\ref{mulan}).
\end{proof}

\begin{lemma}
\label{main}There exist a subsequence $\left\{  n_{j}\right\}  _{j\in N}$ and
a vector $w\in X\cap\mathbb{S}_{Y}$ such that $w_{n_{j}}\rightarrow w$ in $X.$
Moreover, $\left(  \lambda,w\right)  $ is an eigenpair of (\ref{abst}) and
$w_{n_{j}+1}\rightarrow w$ in $X.$
\end{lemma}

\begin{proof}
Since $\left\Vert w_{n}\right\Vert _{X}^{p}\rightarrow\lambda,$ the
compactness of the immersion $X\hookrightarrow Y$ guarantees the existence of
a subsequence $\left\{  w_{n_{j}}\right\}  $ and an element $w\in X$ such that%
\[
w_{n_{j}}\rightharpoonup w\;\mathrm{(weakly)}\;\mathrm{in}\;X,\quad w_{n_{j}%
}\rightarrow w\;\mathrm{(strongly)}\;\mathrm{in}\;Y
\]
and%
\begin{equation}
\left\Vert w\right\Vert _{X}^{p}\leq\lim_{j\rightarrow\infty}\left\Vert
w_{n_{j}}\right\Vert _{X}^{p}=\lambda. \label{3}%
\end{equation}
We also have
\[
\lambda_{n_{j}}\left\langle B(w_{n_{j}}),w\right\rangle =\left\langle
Aw_{n_{j}+1},w\right\rangle \leq\left\Vert w_{n_{j}+1}\right\Vert _{X}%
^{p-1}\left\Vert w\right\Vert _{X}\leq(\lambda_{n_{j}})^{\frac{p-1}{p}%
}\left\Vert w\right\Vert _{X},
\]
so that
\[
\left\langle B(w_{n_{j}}),w\right\rangle \leq(\lambda_{n_{j}})^{-\frac{1}{p}%
}\left\Vert w\right\Vert _{X}\leq\lambda^{-\frac{1}{p}}\left\Vert w\right\Vert
_{X}%
\]
The strong convergence $w_{n_{j}}\rightarrow w$ in $Y$ implies that
$w\in\mathbb{S}_{Y}$ and then the continuity of $B$ yields%
\[
1=\left\Vert w\right\Vert _{Y}^{q}=\left\langle B(w),w\right\rangle
=\lim_{j\rightarrow\infty}\left\langle B(w_{n_{j}}),w\right\rangle \leq
\lambda^{-\frac{1}{p}}\left\Vert w\right\Vert _{X},
\]
so that
\[
\lambda\leq\left\Vert w\right\Vert _{X}^{p}.
\]
This inequality, in view of (\ref{3}), implies that $\lim_{j\rightarrow\infty
}\left\Vert w_{n_{j}}\right\Vert _{X}^{p}=\left\Vert w\right\Vert _{X}%
^{p}=\lambda.$ Hence, the uniform convexity of $X$ allows us to conclude that
$w_{n_{j}}\rightarrow w$ in $X.$ Applying the same arguments to the sequence
$\left\{  w_{n_{j}+1}\right\}  _{j\in\mathbb{N}},$ we can assume that there
exist a subsequence $\left\{  w_{n_{j_{k}}+1}\right\}  _{k\in\mathbb{N}}$ and
a point $\widetilde{w}\in X\cap\mathbb{S}_{Y}$ such that
\[
\left\Vert \widetilde{w}\right\Vert _{X}^{p}=\lambda\quad\mathrm{and}\quad
w_{n_{j_{k}}+1}\rightarrow\widetilde{w}\;\mathrm{in}\;X.
\]
Since $A$ and $B$ are continuous, we can pass to the limit in%
\[
\left\langle A(w_{n_{j_{k}}+1}),v\right\rangle =\lambda_{n_{j_{k}}%
}\left\langle B(w_{n_{j_{k}}}),v\right\rangle ,\quad v\in X,
\]
in order to obtain%
\[
\left\langle A(\widetilde{w}),v\right\rangle =\lambda\left\langle
B(w),v\right\rangle .
\]
This yields%
\[
\lambda=\lambda\left\Vert w\right\Vert _{Y}^{q}=\lambda\left\langle
B(w),w\right\rangle =\left\langle A(\widetilde{w}),w\right\rangle
\leq\left\Vert \widetilde{w}\right\Vert _{X}^{p-1}\left\Vert w\right\Vert
_{X}=\lambda^{\frac{p-1}{p}}\lambda^{\frac{1}{p}}=\lambda,
\]
showing thus that%
\[
\left\langle A(\widetilde{w}),w\right\rangle =\left\Vert \widetilde
{w}\right\Vert _{X}^{p-1}\left\Vert w\right\Vert _{X}.
\]
Since $\left\Vert \widetilde{w}\right\Vert _{X}=\left\Vert w\right\Vert _{X}$
($=\lambda$), our hypothesis \textrm{(A2)} implies that $\widetilde{w}=w,$ so
that%
\[
\left\langle A(w),v\right\rangle =\lambda\left\langle B(w),v\right\rangle
,\quad\forall\mathrm{\;}v\in X.
\]
This shows both that $\lambda$ is an eigenvalue and that $w$ is a
corresponding eigenvector.

Note that our arguments show that $w$ is the only cluster point of the
subsequence $\left\{  w_{n_{j}+1}\right\}  _{j\in\mathbb{N}}.$ This fact
implies that $\left\{  w_{n_{j}+1}\right\}  $ also converges to $w$ as claimed
in the statement of the lemma. Actually, for any $m\in\mathbb{N}$ the shifted
subsequence $\left\{  w_{n_{j}+m}\right\}  _{j\in\mathbb{N}}$ converges to
$w.$ \bigskip
\end{proof}

\begin{proof}
[Proof of Theorem \ref{maintheo}]It follows from Lemma \ref{HyndLind} and
Lemma \ref{main}.
\end{proof}

\begin{remark}
\label{obs}When we know in advance that $\lambda$ is simple, in the sense that
its corresponding eigenvectors are scalar multiple of each other, we have that
$w$ and $-w$ are the only cluster points of the sequence $\left\{
w_{n}\right\}  _{n\in\mathbb{N}}.$
\end{remark}

Regarding the eigenvalue problem (\ref{pq}), when $q\in\lbrack1,p]$ and
$w_{0}\in L^{q}(\Omega)\setminus\left\{  0\right\}  $ is nonnegative, one has%
\[
w_{n}:=\frac{u_{n}}{\left\Vert u_{n}\right\Vert _{q}}\rightarrow e_{q}%
\quad\mathrm{in}\quad W_{0}^{1,p}(\Omega)\quad\mathrm{and}\quad\lambda
_{n}:=\left(  \frac{\left\Vert u_{n}\right\Vert _{Y}}{\left\Vert
u_{n+1}\right\Vert _{Y}}\right)  ^{p-1}\rightarrow\lambda_{q},
\]
where $e_{q}$ denotes the positive first eigenfunction such that $\left\Vert
e_{q}\right\Vert _{q}=1$ and $\lambda_{q}$ is the first eigenvalue for
(\ref{pq}), defined by (\ref{lamdaq}). Indeed, as mentioned in the
Introduction, when $q\in\lbrack1,p]$ the eigenvalue $\lambda_{q}$ is simple
and its eigenfunctions are the only that do not change sign in $\Omega.$
Hence, since $w_{0}\geq0$ a simple comparison principle guarantees that
$w_{n}\geq0$ for all $n\in\mathbb{N}$, the same occurring with the limit
function $w$ given by Theorem \ref{maintheo}. Since $w$ is a nonnegative
eigenfunction corresponding to the eigenvalue $\lambda,$ it must be strictly
positive in $\Omega,$ according to the strong maximum principle (see
\cite{V}). This fact implies that $\lambda=\lambda_{q}$ and then Remark
\ref{obs} guarantees that $w_{n}\rightarrow e_{q}$ in $W_{0}^{1,p}(\Omega).$


\section{Two concrete examples\label{sec2}}

In this section we present two concrete examples of eigenvalue-type problems
for which the results in the previous section apply. In both, $\Omega$ denotes
a smooth bounded domain of $\mathbb{R}^{N},$ $N\geq2.$ We anticipate that when
$1\leq q\leq p$ in both examples the first eigenvalue is simple and its
eigenfunctions are the only that do not change sign. Thus, in this situation,
it follows from the Strong Maximum Principle that the choice of an initial
function $w_{0}$ nonnegative forces $\left\{  \lambda_{n}\right\}  $ to
converge to the first eigenvalue and $\left\{  w_{n}\right\}  $ to converge to
the only positive and normalized first eigenfunction.

\subsection{Dirichlet eigenproblem for the $s$-fractional $p$-Laplacian:}

The results of Section \ref{sec1} can be applied to the following fractional
version of (\ref{pq})%
\begin{equation}
\left\{
\begin{array}
[c]{lll}%
\left(  -\Delta_{p}\right)  ^{s}u=\lambda\left\Vert u\right\Vert _{q}%
^{p-q}\left\vert u\right\vert ^{q-2}u & \mathrm{in} & \Omega\\
u=0 & \mathrm{on} & \partial\Omega,
\end{array}
\right.  \label{fractional}%
\end{equation}
where $0<s<1<p,$ $\left\Vert \cdot\right\Vert _{q}$ denotes the standard norm
of $L^{q}(\Omega),$
\[
1\leq q<p_{s}^{\star}:=\left\{
\begin{array}
[c]{lll}%
\frac{Np}{N-ps} & \mathrm{if} & sp<N\\
\infty & \mathrm{if} & sp\geq N
\end{array}
\right.
\]
and
\[
\left(  -\Delta_{p}\right)  ^{s}u:=2\lim_{\epsilon\rightarrow0^{+}}%
\int_{\left\vert x\right\vert \geq\epsilon}\frac{|u(x)-u(y)|^{p-2}%
(u(x)-u(y))}{|x-y|^{N+sp}}\mathrm{d}y,
\]
is the $s$-fractional $p$-Laplacian.

The usual space to deal with this problem is the fractional Sobolev space
$W_{0}^{s,p}(\Omega)$ defined as the closure of $C_{c}(\Omega)$ with respect
to the Gagliardo seminorm $\left[  \cdot\right]  _{s,p}$ in $\mathbb{R}^{N},$
whose expression, at a measurable function $u$ of $\mathbb{R}^{N},$ is%
\[
\left[  u\right]  _{s,p}:=\left(  \int_{\mathbb{R}^{N}}\int_{\mathbb{R}^{N}%
}\frac{|u(x)-u(y)|^{p}}{|x-y|^{N+sp}}\mathrm{d}x\mathrm{d}y\right)  ^{\frac
{1}{p}}.
\]

Thanks the fractional Poincar\'{e} inequality (see \cite[Lemma 2.4]{BLP}) the
Gagliardo seminorm $\left[  \cdot\right]  _{s,p}$ is really a norm in
$W_{0}^{s,p}(\Omega).$

It is well-known that $W_{0}^{s,p}(\Omega),$ endowed with the norm $\left[
\cdot\right]  _{s,p},$ is a Banach space uniformly convex. Moreover,
$W_{0}^{s,p}(\Omega)$ is compactly embedded into $L^{r}(\Omega),$ for all
$1\leq r<p_{s}^{\star}.$

The weak formulation of (\ref{fractional}) is $A(u)=\lambda\left\Vert
u\right\Vert _{q}^{p-q}B(u)$ where $A:W_{0}^{s,p}(\Omega)\rightarrow
W_{0}^{-s,p^{\prime}}(\Omega)$ is defined as%
\begin{equation}
\left\langle A(u),v\right\rangle :=\int_{\mathbb{R}^{N}}\int_{\mathbb{R}^{N}%
}\frac{\left\vert u(x)-u(y)\right\vert ^{p-2}\left(  u(x)-u(y)\right)  \left(
v(x)-v(y)\right)  }{\left\vert x-y\right\vert ^{N+ps}}\mathrm{d}x\mathrm{d}y,
\label{Af}%
\end{equation}
and $B:L^{q}(\Omega)\rightarrow L^{q\prime}(\Omega)$ is the map%
\begin{equation}
\left\langle B(u),v\right\rangle ={\int_{\Omega}}\left\vert u\right\vert
^{q-2}uv\mathrm{d}x. \label{Bf}%
\end{equation}

Therefore, by considering these maps, (\ref{fractional}) takes the form
(\ref{absol}) with $X=W_{0}^{s,p}(\Omega)$ and $Y=L^{q}(\Omega).$ It can be
shown that the functions (\ref{Af}) and (\ref{Bf}) satisfy the hypotheses
(\textrm{A1}), (\textrm{A2}), (\textrm{B1}), (\textrm{B2}) and (\textrm{AB}).
The proof of these claims as well as all of that made in this section on the
fractional Sobolev space $W_{0}^{s,p}(\Omega)$ and the operator $\left(
-\Delta_{p}\right)  ^{s}$ can be found in the papers \cite{BLP, BF,Guide,
ILPS, Erilind}.

\subsection{Steklov eigenproblem for the $p$-Laplacian}

Let us consider the following Steklov-type eigenvalue problem
\begin{equation}
\left\{
\begin{array}
[c]{lll}%
-\Delta_{p}u+\left\vert u\right\vert ^{p-2}u=0 & \mathrm{in} & \Omega\\
\left\vert \nabla u\right\vert ^{p-2}\frac{\partial u}{\partial\nu}%
=\lambda\left\vert u\right\vert _{q}^{p-q}\left\vert u\right\vert ^{q-2}u &
\mathrm{on} & \partial\Omega,
\end{array}
\right.  \label{stek}%
\end{equation}
where $\frac{\partial}{\partial\nu}$ denotes the outer unit normal derivative
along $\partial\Omega$ and
\[
\left\vert u\right\vert _{q}:=\left(
{\displaystyle\int_{\partial\Omega}}
\left\vert u\right\vert ^{q}\mathrm{d}s\right)  ^{\frac{1}{q}}%
\]
denotes the standard norm of the Banach space $L^{q}(\partial\Omega).$

The appropriate space of solutions for (\ref{stek}) is the uniformly convex
Sobolev space
\[
W^{1,p}(\Omega):=\left\{  u\in L^{p}(\Omega):\nabla u\in L^{p}(\Omega
)^{N}\right\}
\]
endowed with the norm%
\[
\left\Vert u\right\Vert _{W^{1,p}}:=\left(  \left\Vert \nabla u\right\Vert
_{p}^{p}+\left\Vert u\right\Vert _{p}^{p}\right)  ^{\frac{1}{p}}.
\]

The embedding $W^{1,p}(\Omega)\hookrightarrow L^{q}(\partial\Omega)$ is known
as the \textit{boundary trace operator} and associates $u\in W^{1,p}(\Omega)$
with its trace $\left.  u\right\vert _{\partial u}\in L^{q}(\Omega),$ which we
denote here by $u$ itself. This operator is compact if
\[
1\leq q<p_{\ast}:=\left\{
\begin{array}
[c]{lll}%
\frac{p(N-1)}{N-p} & \mathrm{if} & p<N\\
\infty & \mathrm{if} & p\geq N
\end{array}
\right.
\]
and just continuous when $q=p_{\star}$ (see \cite{BFR, Rossi}).

We say that $u\in W^{1,p}(\Omega)$ is a weak solution of (\ref{stek}), for
some $\lambda\in\mathbb{R},$ if and only if,%
\[
\int_{\Omega}\left(  \left\vert \nabla u\right\vert ^{p-2}\nabla u\cdot\nabla
v+\left\vert u\right\vert ^{p-2}uv\right)  \mathrm{d}x=\lambda\left\vert
u\right\vert _{q}^{p-q}{\int_{\partial\Omega}}\left\vert u\right\vert
^{q-2}uv\mathrm{d}s,\quad\forall\mathrm{\;}v\in W^{1,p}(\Omega).
\]

Thus, (\ref{stek}) takes the form (\ref{absol}) with $X:=(W^{1,p}%
(\Omega),\left\Vert u\right\Vert _{W_{0}^{1,p}}),$ $Y:=(L^{q}(\partial
\Omega),\left\vert u\right\vert _{q})$ and the maps $A:X\rightarrow X^{\star}$
and $B:Y\rightarrow Y^{\star}$ defined by%
\[
\left\langle A(u),v\right\rangle :=\int_{\Omega}\left\vert \nabla u\right\vert
^{p-2}\nabla u\cdot\nabla v\mathrm{d}x+\int_{\Omega}\left\vert u\right\vert
^{p-2}uv\mathrm{d}x,\quad\forall\mathrm{\;}v\in W^{1,p}(\Omega)
\]
and%
\[
\left\langle B(u),v\right\rangle :={\int_{\partial\Omega}}\left\vert
u\right\vert ^{q-2}uv\mathrm{d}s,\quad\forall\mathrm{\;}v\in L^{q}%
(\partial\Omega).
\]

It is straightforward to check (see \cite{Le}) that $A$ and $B$ are continuous
and satisfy the hypotheses $\mathrm{(A1)}$, $\mathrm{(A2)}$, $\mathrm{(B1),}$
$\mathrm{(B2)}$ and $\mathrm{(AB)}$.

For more details on the eigenvalue problem (\ref{stek}) we refer the reader to
\cite{BR} (see also \cite{Auchmuty} where properties and applications
regarding the case $p=q=2$ are provided).

\section{acknowledgements}

The author was supported by CNPq/Brazil (483970/2013-1 and 306590/2014-0) and
Fapemig/Brazil (CEX APQ 03372/16).

\end{document}